\documentclass[reqno]{amsart}
\usepackage[utf8]{inputenc}
\usepackage{amsmath,amsthm,amssymb}
\usepackage{tikz}
\usepackage{tikz-cd}
\usepackage{mathabx}

\usepackage[hmarginratio=1:1]{geometry}

\usepackage[english]{babel}

\usepackage[hidelinks]{hyperref}
\usepackage{orcidlink}

\usepackage{wasysym, stackengine, makebox, graphicx}

\renewcommand{\phi}{\varphi}
\renewcommand{\kappa}{\varkappa}
\renewcommand{\epsilon}{\varepsilon}

\newtheorem{lemma}{Lemma}
\newtheorem{corollary}{Corollary}
\newtheorem{theorem}{Theorem}
\newtheorem{axiom}{Axiom}

\newcommand\ar[3]
{
  \!\!\ensuremath
  {
    \begin{tikzcd}[ampersand replacement=\&]
      #1: #2 \ar[r] \& #3
    \end{tikzcd}
  }\!\!
}

\newcommand\artxt[4]
{
  \!\!\ensuremath
  {
    \begin{tikzcd}[ampersand replacement=\&]
      #1: #2 \ar[r,"#4"] \& #3
    \end{tikzcd}
  }\!\!
}

\DeclareMathOperator{\Ker}{Ker}
\DeclareMathOperator{\im}{im}
\DeclareMathOperator{\Img}{Im}

\DeclareMathOperator{\coker}{coker}
\DeclareMathOperator{\Coker}{Coker}
\DeclareMathOperator{\coim}{coim}
\DeclareMathOperator{\Coim}{Coim}

\DeclareFontFamily{U}{BOONDOX-cal}{\skewchar\font=45 }
\DeclareFontShape{U}{BOONDOX-cal}{m}{n}{
  <-> s*[1.05] BOONDOX-r-cal}{}
\DeclareFontShape{U}{BOONDOX-cal}{b}{n}{
  <-> s*[1.05] BOONDOX-b-cal}{}
\DeclareMathAlphabet{\mathcal}{U}{BOONDOX-cal}{m}{n}
\SetMathAlphabet{\mathcal}{bold}{U}{BOONDOX-cal}{b}{n}
\DeclareMathAlphabet{\mathbcalb}{U}{BOONDOX-cal}{b}{n}

\title[Some Properties of Homology and Exactness of Nomura's Homology Sequences]{Some Properties of Homology and Exactness of Nomura's Homology Sequences in a Grandis Homological Category}
\thanks{ }

\author{Yaroslav Kopylov}
\address{Yaroslav Kopylov\,
\orcidlink{0000-0002-0343-4424}\,
0000-0002-0343-4424
\newline\hphantom{iii} Sobolev Institute of Mathematics
\newline\hphantom{iii} 4 Koptyug Ave.
\newline\hphantom{iii} 630090, Novosibirsk, Russia}
\email{\href{mailto:yakop@math.nsc.ru}{yakop@math.nsc.ru}}

\author{Vadim Leshkov}
\address{Vadim Leshkov\,
\orcidlink{0009-0000-2768-9942}\,
0009-0000-2768-9942
\newline\hphantom{iii} University of Virginia
\newline\hphantom{iii} P. O. Box 400137
\newline\hphantom{iii} Department of Mathematics
\newline\hphantom{iii} 127 Kerchof Hall
\newline\hphantom{iii} Charlotsville, VA 22904, USA}
\email{\href{mailto:frj5jd@virginia.edu}{frj5jd@virginia.edu}}

\begin{document}

\begin{abstract}
We consider Lambek's invariants $\Ker$ and $\Img$ for commutative squares in Grandis homological categories. We prove that Nomura's null sequences exist in such categories and find sufficienct conditions for their exactness. We also prove the coincidence of left and right homology in  Grandis homological categories.

\vspace{2mm}
\noindent
\textbf{Key words and phrases:} commutative square, Grandis homological category, homology, Nomura's null sequences, exactness

\vspace{2mm}
\noindent
\textbf{Mathematics Subject Classification 2020:} 18G50
\end{abstract}

\maketitle 

\section{Introduction}\label{sec_intro}

In 1964, in~\cite{Lambek-1964}, Lambek proved the~following assertion:
\smallskip

{\it Given a~commutative diagram}
\begin{equation}\label{maind}
\begin{tikzcd}
    A \ar[r,"f"] \ar[d,"a"'] \ar[r, phantom, shift right=4ex, "S" marking] & B \ar[r,"g"] \ar[d,"b" description] \ar[r, phantom, shift right=4ex, "T" marking] & C \ar[d,"c"]\\
    A' \ar[r,"f'"'] & B' \ar[r,"g'"'] & C'
\end{tikzcd}
\end{equation}
{\it of groups and group homomorphisms with exact rows, there is a~natural
isomorphism
$$
(\Img b \cap \Img f')/\Img(bf)
\cong
\Ker(cg)/(\Ker b\cdot \Ker g).
$$ }

For a~commutative
square
\begin{equation}
\label{comm_square}
\begin{tikzcd}
    A \ar[r,"f"] \ar[d,"a"'] \ar[r, phantom, shift right=4ex, "S" marking] & B \ar[d,"b"] \\
    C \ar[r,"g"'] & D
\end{tikzcd}
\end{equation}
in the~category of groups, we put:
$$
\Img S:=(\Img b\cap \Img g)/\Img(bf), \quad
\Ker S:=\Ker(bf)/(\Ker a \cdot \Ker f).
$$

Later Leicht~\protect\cite{Le} extended Lambek's theorem to a~wider class of~categories that includes the~category of~groups~$\mathcal{G\!r\!p}$, thus giving a~homological proof.
In~\protect\cite{No1}, Nomura proved Lambek's isomorphism in~a~Puppe exact category.
In~\cite{Ko05_2}, these invariants were considered in~a~quasi-abelian category, and Lambek's isomorphism was established provided that the~morphism~$b$ is exact. In~\cite{KoLe2024}, a~commutative diagram of two commutative squares~\eqref{maind} was considered in a semiexact category in the sense of M.~Grandis and it was shown that there exists a~natural morphism $\Lambda_{S,T} : \Img S \to \Ker T$ (Lambek's morphism) provided that the~morphism~$b$ is exact and rows are null-sequences which are not necessarily exact.

In~\cite{No1}, Y.~Nomura constructed two exact sequences in a~Puppe exact category corresponding to diagram~\eqref{maind}. The~first sequence looks as
\begin{equation}\label{diagram_nomura_exact_seq}
\begin{tikzcd}[column sep=5pt]
    & 0 \ar[r]  \ar[d, phantom, ""{coordinate, name=Z}] & H\left( \Ker(bf) \to \Ker b \to \Ker c \right) \ar[r] & \Ker\left( H \to H' \right) \ar[r] & \Img S 
    \ar[
        dlll,
        "\Lambda"',
        rounded corners,
        to path={ -- ([xshift=2ex]\tikztostart.east)
        |- (Z) [near end]\tikztonodes
        -| ([xshift=-2ex]\tikztotarget.west) 
        -- (\tikztotarget)}
    ]\\
   & \Ker T \ar[r] & \Coker\left( H \to H' \right) \ar[r]  & H\left(\Coker a \to \Coker b \to \Coker(g' b) \right) \ar[r]  & 0
\end{tikzcd}
\end{equation}  
where $\Lambda = \Lambda_{S,T} : \Img S \to \Ker T$ is the~Lambek morphism, $H$ is the homology of the first row, and $H'$ is the homology of the second row.

The~second sequence is 
\begin{equation}\label{seq-2}
\begin{tikzcd}[column sep=small]
    \ar[d, phantom, ""{coordinate, name=Z}] \Ker S \ar[r,"p_1"] & H(\Ker a\to \Ker b\to \Ker c) 
    \ar[r,"\varkappa"] & \Ker\left( H \to H' \right) \ar[r, "\beta"] & \Img S 
    \ar[
        dlll,
        "\Lambda"',
        rounded corners,
        to path={ -- ([xshift=2ex]\tikztostart.east)
        |- (Z) [near end]\tikztonodes
        -| ([xshift=-2ex]\tikztotarget.west) 
        -- (\tikztotarget)}
    ]
    \\
    \Ker T \ar[r,"\beta'"'] & \Coker\left( H \to H' \right) \ar[r,"\varkappa'"'] & H( \Coker a\to \Coker b\to \Coker c) \ar[r,"p_1'"']
    & \Img T
\end{tikzcd}
\end{equation}

In~the~present article, we prove the existence of sequences~\eqref{diagram_nomura_exact_seq} and \eqref{seq-2} and obtain criteria for their exactness in~the~framework of~Grandis homological categories (see~\cite{Gr92,MGrandis-2013}).
These categories were introduced as a~tool for~developing
homological algebra ``in~a~strongly non-abelian setting.'' In~recent years, Grandis homological categories have found applications in~homological algebra and the~theory of approximate representations of groups~\cite{ConnCons2019,Fr2023}.

The~structure of the~article is as follows.
In~Section~\ref{sec_hom_th_hom_cat}, we recall the~basic definitions concerning semiexact (ex1-categories), ex2-, and homological categories and prove that an ex2-category is homological if and only if the~natural morphism from the~left homology object to the~right homology object is an~isomorphism
(a~generalization of~\cite[Theorem~3]{Ko_semr12}, which is a~criterion for a~Raikov semi-abelian category to be quasi-abelian). In~Section~\ref{sec-lamb}, we recall the~construction of Lambek's morphism~$\Lambda$ and some related facts. Section~\ref{nomura_null} is devoted to constructing and investigating Nomura's null sequences~\eqref{diagram_nomura_exact_seq} and \eqref{seq-2} in a~Grandis homological category; some related assertions are also proved.

\section{Homology in a Homological Category}\label{sec_hom_th_hom_cat}

\subsection{Semiexact Categories}

We recall the~notion of an ideal in a~category and the~notions of kernel and cokernel with respect to an ideal (see \cite[1.3.1]{MGrandis-2013}).

Let $\mathcal{C}$ be a category.
Consider a class of morphisms $\mathcal{N} \subseteq {\rm Mor}(\mathcal{C})$ in $\mathcal{C}$.
The~class $\mathcal{N}$ is called an \emph{ideal} in $\mathcal{C}$ if for every morphism $f \in \mathcal{N}$ any legitimate composite $hfg$ belongs to $\mathcal{N}$.
A pair $(\mathcal{C}, \mathcal{N})$ consisting of a category $\mathcal{C}$ equipped with an ideal $\mathcal{N}$ is called a \emph{category with null morphisms} and morphisms from $\mathcal{N}$ are called \emph{null morphisms}.
We denote the~fact that a morphism $f: A \to B$ is null by the~following label over the~arrow:
$$
\artxt{f}{A}{B}{{\rm null}}
\ \overset{{\rm def}}{\Longleftrightarrow}\
f \in \mathcal{N}.$$
An object $A \in {\rm Ob}(\mathcal{C})$ is called a \emph{null object} if ${\rm id}_A \in \mathcal{N}$.
An ideal $\mathcal{N}$ is said to be \emph{closed} if there exists a class $\mathcal{O} \subseteq {\rm Ob}(\mathcal{C})$ such that every morphism from $\mathcal{N}$ factors through some object in $\mathcal{O}$.

In a category with null morphisms $(\mathcal{C}, \mathcal{N})$, a~morphism
$k: K \to A$ is called a \emph{kernel} of a morphism $f: A \to B$ if $fk\in \mathcal{N}$ and for any morphism $x$ with $fx\in \mathcal{N}$ there exists a~unique morphism $x'$ such that $x=kx'$. Dually, a~morphism
$l: B \to L$ is called a \emph{cokernel} of a morphism $f: A \to B$ if $lf\in \mathcal{N}$ and for any morphism $y$ with $yf\in \mathcal{N}$ there exists a~unique morphism $y'$ such that $y=y'l$. Clearly, two kernels (cokernels) of any morphism are isomorphic. If $k: K \to A$ is a~kernel of $f: A \to B$ then we write $K = \Ker f$ and $k = \ker f$; if $l: L \to B$ is a~cokernel of~$f$ then we write $L=\Coker f$ and
$l=\coker f$. All we have said is schematically depicted in~the~diagram
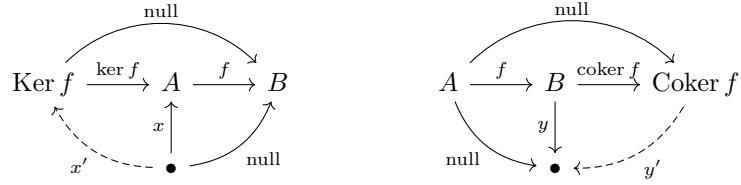
\[
\begin{tikzcd}
    \Ker f \ar[r,"\ker f"] \ar[rr,bend left=45,"{\rm null}"] & A \ar[r,"f"] & B & & A \ar[r,"f"] \ar[rr,"{\rm null}", bend left=45] \ar[dr,"{\rm null}"',bend right=30] & B \ar[r,"\coker f"] \ar[d,"y"'] & \Coker f \ar[dl,dashed,bend left=30,"y'"] 
  \\
   & \bullet \ar[ul,dashed,bend left=30,"x'"] \ar[u,"x"] \ar[ur,"{\rm null}"',bend right=30] & & & & \bullet
\end{tikzcd}
\]
    
The~\emph{image} and \emph{coimage} of a morphism $f: A \to B$ are defined as 
$$\ar{\im f := \ker(\coker f)}{\Img f}{B}$$
and
$$\ar{\coim f := \coker(\ker f)}{A}{\Coim f},$$
respectively.
We say that a~morphism $f$ is an~$\mathcal{N}$\emph{-monomorphism} (respectively, $\mathcal{N}$\emph{-epimorphism}) if it satisfies the~left-hand condition (respectively, the~right-hand condition):
$$
\text{if}\ fh\ \text{is null then}\ h\ \text{is null};
\quad\quad
\text{if}\ kf\ \text{is null then}\ k\ \text{is null}.
$$
It is clear that $\ker f$ is an~$\mathcal{N}$-monomorphism and $\coker f$ is an $\mathcal{N}$-epimorphism.
The~arrows $\rightarrowtail$ and $\twoheadrightarrow$ are reserved for kernels and cokernels respectively.
Note that if $(f: X \to Y) \in \mathcal{N}$ then $\ker f = {\rm id}_X$ and $\coker f = {\rm id}_Y$.
We make use the~following important lemma (see \cite[1.3.1]{MGrandis-2013}):

\begin{lemma}\label{lemma_n_mono_ker_null_and_dual}
    The~following are equivalent:
    \begin{itemize}
        \item[{\rm (i)}]
        $f$ is an $\mathcal{N}$-monomorphism;
        \item[{\rm (ii)}]
        $\ker f \in \mathcal{N}$;
        \item[{\rm (iii)}]
        $\ker(fh) = \ker h$ for any $h$.
    \end{itemize}
    Dually, the~following are equivalent:
    \begin{itemize}
        \item[{\rm (i)}]
        $f$ is an $\mathcal{N}$-epimorphism;
        \item[{\rm (ii)}]
        $\coker f \in \mathcal{N}$;
        \item[{\rm (iii)}]
        $\coker(kf) = \coker k$ for any $k$.
    \end{itemize}
\end{lemma}

Let $(\mathcal{C}, \mathcal{N})$ be a category with null morphisms.
Consider the~following two axioms~\cite{MGrandis-2013}.

\begin{axiom}\label{axiom_ex0}
    The~ideal $\mathcal{N}$ is closed.
\end{axiom}

\begin{axiom}\label{axiom_ex1}
    Every morphism has a~kernel and a~cokernel with respect to $\mathcal{N}$.
\end{axiom}

A category with null morphisms is called \emph{semiexact} if it satisfies Axioms \ref{axiom_ex0} and \ref{axiom_ex1}.
Semiexact categories are also called \emph{ex1-categories}.

It is easy to see that for any $\mathcal{N}$-monomorphism (respectively,
any $\mathcal{N}$-epimorphism) $f$ in a~semi-exact category,  $\Ker f$ (respectively, $\Coker f$) is a~null object.

\begin{lemma}
In a semiexact category $(\mathcal{C}, \mathcal{N})$,
if $\alpha\colon A \twoheadrightarrow B$ is a~cokernel,
i.e. $\alpha = \coker \alpha'$
for a~morphism $\alpha'\colon \bullet \to A$,
then for every $g\colon A \to A'$ there exists
a~morphism $g'\colon B \to \Coker(g \alpha')$
such that the~left-hand square is a~pushout:
\[
\begin{tikzcd}
    \bullet \ar[r,"\alpha'"] & A \ar[r,"\alpha\, = \coker \alpha'",two heads] \ar[d,"g"'] \ar[r, phantom, shift right=4ex, "{\rm PO}" marking] & B \ar[d,"g'",dashed]\\
    & A' \ar[r, "\coker(g \alpha')"',two heads] & \Coker(g \alpha')
\end{tikzcd}
\quad\quad
\quad\quad
\begin{tikzcd}
    \Ker(\beta' f) \ar[r,"\ker (\beta' f)",tail] \ar[d,dashed,"f'"'] & D' \ar[d,"f"]\\
    C \ar[r, phantom, shift left=4ex, "{\rm PB}" marking] \ar[r,"\beta\, = \ker \beta'"',tail] & D \ar[r,"\beta'"'] & \bullet
\end{tikzcd}
\]
Dually, if $\beta\colon C \rightarrowtail D$ is a~kernel,
i.e. $\beta = \ker \beta'$ with $\beta'\colon D \to \bullet$,
then for every $f\colon D' \to D$ there exists
a~morphism $f'\colon \Ker(\beta' f) \to C$
such that the~right-hand square is a~pullback.
\end{lemma}

Consider a morphism $f: A \to B$ in a semiexact category $(\mathcal{C}, \mathcal{N})$.
Note that $(\coker f) f$ is null.
Hence, there exists a unique morphism $f': A \to \Img\, f$ such that $f = (\im f) f'$.
Note that $f' \ker f$ is null.
Hence, $f'$ factors through the~coimage of $f$, i.e. there exists a morphism $\overline{f}: \Coim f \to \Img f$.
\[
\begin{tikzcd}
    \Ker f \ar[r,"\ker f",tail] & A \ar[r,"f"] \ar[d,"\coim f"',two heads] \ar[rd,"f'"',dashed,bend left=25] & B \ar[r,"\coker f",two heads] & \Coker f\\
    & \Coim f \ar[r,dashed,"\overline{f}"'] & \Img f \ar[u,"\im f"',tail]
\end{tikzcd}
\]
Therefore, any morphism $f$ in a semiexact category can be factored as 
\begin{equation}\label{normal-fact}
f = (\im f)\, \overline{f} \coim f. 
\end{equation}
We refer to~\eqref{normal-fact} as a~normal (or canonical) decomposition
of~$f$.
Two normal decompositions of~$f$ are canonically isomorphic.
The~morphism~$\overline{f}$ in~\eqref{normal-fact}
is defined uniquely once $\im f$ and $\coim f$ are chosen.
If $\overline{f}$ is an isomorphism then morphism $f$ is called \emph{exact}.
Usually, when morphism $f$ is exact, we omit the~middle isomorphism and write $f = (\im f) \coim f$.

Consider a commutative square (\ref{comm_square}) in a semiexact category.
Note that
$\beta g \ker \alpha = f \alpha \ker \alpha \in \mathcal{N}$,
and so there exists a unique morphism $\widehat{g}\colon \Ker \alpha \to \Ker \beta$ 
such that $(\ker\beta)\widehat{g}=g\ker\alpha$, which we call \emph{the~morphism~of~kernels~induced~by~$g$}. 
Dually, there exists a unique morphism of cokernels 
$\widecheck{f}:\Coker\alpha \to \Coker\beta$ such that
$\widecheck{f}\coker\alpha= (\coker \beta)f$, which we call 
\emph{the~morphism~of~cokernels~induced~by~$f$}.
$$
\begin{tikzcd}
\Ker \alpha \ar[r,"\ker\alpha",tail] \ar[d,dashed,"\widehat{g}"'] & A 
\ar[r,"\alpha"] \ar[d,"g"'] & B \ar[r,"\coker \alpha",two heads] \ar[d,"f"] & \Coker \alpha 
\ar[d,"\widecheck{f}",dashed] \\
\Ker \beta \ar[r,"\ker \beta"',tail] & C 
\ar[r,"\beta"'] & D \ar[r,"\coker \beta"',two heads] & \Coker \beta
\end{tikzcd}
$$
This notation is somewhat ambiguous since $f$ can be included in different commutative squares but in what follows, the~square under consideration always be clear from the~context.

It is easy to see that kernels and cokernels satisfy the~isomorphism identities described in the~following important lemma formulated in~\cite{No1} and easy to prove.

\begin{lemma}[$3 \times 3$ lemma]\label{lemma_three_by_three}
    For a~commutative square
    \[
    \begin{tikzcd}
        A \ar[d,"a"'] \ar[r,"f"] & B \ar[d,"b"] \\
        C \ar[r,"g"'] & D
    \end{tikzcd}
    \]
     in an~semiexact category, there exist unique isomorphisms:
    \begin{itemize}
        \item[{\rm (i)}]
        $
        \begin{tikzcd}
            \lambda : \Ker(\Ker a \xrightarrow[]{\widehat{f}} \Ker b) \ar[r,"\cong"] & \Ker(\Ker f \xrightarrow[]{\widehat{a}} \Ker g)
        \end{tikzcd}
        $,
        \item[{\rm (ii)}]
        $
        \begin{tikzcd}
            \mu : \Coker(\Coker a \xrightarrow[]{\widecheck{g}} \Coker b) \ar[r,"\cong"] & \Coker(\Coker f \xrightarrow[]{\widecheck{b}} \Coker g)
        \end{tikzcd}
        $
    \end{itemize}
    such that
    $(\ker a) \ker \widehat{f} = (\ker f) (\ker \widehat{a}) \lambda$
    and
    $\mu (\coker \widecheck{g}) \coker b = (\coker \widecheck{b}) \coker g$.
\end{lemma}

The~following assertion is Lemma~2.3 in~\cite{No1}.  It is easy to see that it holds in~any semiexact category; we omit its straightforward proof.

\begin{lemma}\label{PB}
Let
\begin{equation}\label{l2.3}
\begin{tikzcd}
\bullet \arrow[r, "b"] \arrow[d, "a"'] & \bullet \arrow[d, "c"] \\
\bullet \arrow[r, "d"']                & \bullet                
\end{tikzcd}
\end{equation}
be a~commutative square and $c$ be an~$\mathcal{N}$-monomorphism. Let
$\widehat{a}$ be the~morphism of the~kernels of the~rows of~\eqref{l2.3},
i.e., the~morphism such that $a\ker\,b=(\ker\,d)\widehat{a}$. Then the~square
$$
\begin{tikzcd}
\bullet \arrow[r, "\ker\,b"] \arrow[d, "\widehat{a}"'] & \bullet \arrow[d, "a"] \\
\bullet \arrow[r, "\ker\,d"']                & \bullet                
\end{tikzcd}
$$
is a~pullback.
\end{lemma}

The~following lemma is \cite[Lemma~2.5 and 2.7]{KoLe2024}.

\begin{lemma}[]\label{lemma_kop_weg}
In a semiexact category $(\mathcal{C}, \mathcal{N})$ the~following assertions hold:
\begin{enumerate}
    \item[(i)]
    if kernels are stable under composition then $\overline{f}$ is an~$\mathcal{N}$-epimorphism for any morphism $f$; moreover, if $hg$ is a~kernel then $g$ is a~kernel; 
    
    \item[(ii)]
    dually, if cokernels are stable under composition then $\overline{f}$ is an $\mathcal{N}$-monomorphism for any morphism $f$;  moreover, if $hg$ is a~cokernel then $h$ is a~cokernel.
\end{enumerate}
\end{lemma}

Consider the following axiom.

\begin{axiom}\label{axiom_ex2}
    Kernels and cokernels are closed under composition.
\end{axiom}

A semiexact category $(\mathcal{C}, \mathcal{N})$ is called an~\emph{ex2-category} if it satisfies Axiom~\ref{axiom_ex2}.
In the present work we make use of the next proposition which follows from Lemma~\ref{lemma_kop_weg}.

\begin{corollary}[]\label{corr_nbimorph_in_the_middle}
In an~ex2-category
the morphism $\overline{f}$ is an~$\mathcal{N}$-bimorphism, i.e. an~$\mathcal{N}$-monomorphism and $\mathcal{N}$-epimorphism, for any morphism $f$.
\end{corollary}

\subsection{Homological categories}

In a semiexact category $(\mathcal{C}, \mathcal{N})$ a sequence of morphisms
\begin{equation}\label{diag_short_null_sequence}
\begin{tikzcd}
    X \ar["f",r] & Y \ar["g",r] & Z
\end{tikzcd}
\end{equation}
is called a~\emph{short null-sequence} if $g f \in \mathcal{N}$.
A sequence of morphisms
\begin{equation}\label{diag_long_null_sequence}
\begin{tikzcd}
    \ldots \ar[r,"f_{-2}"] & X_{-1} \ar[r,"f_{-1}"] & X_0 \ar[r,"f_0"] & X_1 \ar[r,"f_1"] & X_2 \ar[r,"f_2"] & \ldots
\end{tikzcd}
\end{equation}
in $(\mathcal{C}, \mathcal{N})$ is called a \emph{null-sequence} if any subsequence
$
\begin{tikzcd}
    X_i \ar["f_i",r] & X_{i+1} \ar["f_{i+1}",r] & X_{i+2}
\end{tikzcd}
$
is a short null-sequence.
The~\emph{left homology} (dually, \emph{right homology}) of
the~null-sequence~\eqref{diag_short_null_sequence} is
an object $H_{-}(f,g)$ (dually, $H_{+}(f,g)$) which is defined as
the cokernel $\Coker \sigma$ of the universal morphism
$\sigma: X\to\Ker g$ such that
$(\ker g) \sigma=f$
(dually, the kernel $\Ker \tau$ of the universal morphism
$\tau: \Coker f\to Z$ such that
$g=\tau \coker f$):
\begin{equation}\label{diag_homology_explained}
\begin{tikzcd}
    & \Ker g \ar[d,tail,"\ker g"] \ar[r,two heads,"\coker \sigma"] & H_{-}(f,g) \\
    X \ar[r,"f"] \ar[ru,bend left=15,dashed,"\sigma"] & Y \ar[r,"g"'] \ar[d,"\coker f"',two heads] & Z \\
    H_{+}(f,g) \ar[r,tail,"\ker \tau"'] & \Coker f \ar[ru,bend right=15,"\tau"',dashed] &
\end{tikzcd}
\end{equation}

\begin{lemma}    \label{lemma_m-morphism}
    In a~semiexact category $\mathcal{C},\mathcal{N}$, for null-sequence~\eqref{diag_short_null_sequence} there exists a unique morphism
    $$m(f,g) : H_{-}(f,g) \longrightarrow H_{+}(f,g)$$
    such that $(\ker \tau)\, m(f,g) \coker \sigma = (\coker f) \ker g$.
\end{lemma}
\begin{proof}
Note that in~diagram~\eqref{diag_homology_explained} we have $(\coker f) (\ker g) \sigma \in \mathcal{N}$.
Hence, there exists a~unique 
$$k : H_{-}(f,g) \longrightarrow \Coker f$$
such that $(\coker f) (\ker g) \sigma = k \coker \sigma$.
Note that $\tau k \coker \sigma \in \mathcal{N}$ and hence $\tau k \in \mathcal{N}$.
Thus, there exists a~unique $m : H_{-}(f,g) \to H_{+}(f,g)$ such that $(\ker \tau) m = k$.
\[
\begin{tikzcd}
    & \Ker g \ar[d,tail,"\ker g"] \ar[r,two heads,"\coker \sigma"] & H_{-}(f,g) \ar[ddl,"k",dashed,bend left=100] \ar[lldd,bend right=100,dashed,"m"'] \\
    X \ar[r,"f"'] \ar[ru,bend left=10,dashed,"\sigma"] & Y \ar[r,"g"] \ar[d,"\coker f"',two heads] & Z \\
    H_{+}(f,g) \ar[r,tail,"\ker \tau"'] & \Coker f \ar[ru,bend right=10,"\tau"',dashed] &
\end{tikzcd}
\]
Note that $m$ is the~unique morphism such that $(\ker \tau) m \coker \sigma = (\coker f) \ker g$.
Hence, $m$ is the~desired morphism and
we put $m(f,g) := m$.
\end{proof}

As in \cite[1.5.2]{MGrandis-2013},
we call sequence (\ref{diag_long_null_sequence})
\emph{exact at $X_i$} if $\im f_{i-1} = \ker f_i$, or, equivalently,
$\coker f_{i-1} = \coim f_i$ due to the~equalities
$$\coker f_{i-1} = \coker(\im f_{i-1}) = \coker(\ker f_{i}) = \coim f_i.$$
For example, a sequence of the form
\[
\begin{tikzcd}
A \ar[r,"f"] & B \ar[r,"\mathrm{null}"] & C
\end{tikzcd}
\quad\quad
\left(\text{respectively, }
\begin{tikzcd}
A' \ar[r,"\mathrm{null}"] & B' \ar[r,"g"] & C'
\end{tikzcd}
\right)
\]
is exact when $\coker f\in\mathcal{N}$ 
(respectively,
$\ker g\in\mathcal{N}$).

\smallskip

Consider the~following axiom \cite{MGrandis-2013}.

\begin{axiom}\label{axiom_ex3}
    Consider a kernel
    $i: A \rightarrowtail B$ and a cokernel
    $q: B \twoheadrightarrow C$
    such that $(\coker i) \ker q \in \mathcal{N}$.
    Then there exist a cokernel
    $\pi: A \twoheadrightarrow H$
    and a kernel
    $\iota: H \rightarrowtail C$
    which make the~square
    \[
    \begin{tikzcd}
        A \ar[d,two heads,dashed,"\pi"'] \ar[r,tail,"i"] & B \ar[d,two heads,"q"]\\
        H \ar[r,tail,dashed,"\iota"'] & C
    \end{tikzcd}
    \]
commute.
\end{axiom}
The~object $H$ in~Axiom~\ref{axiom_ex3} is called a~\emph{homology object}.
An ex2-category is called  \emph{homological} (or an \emph{ex3-category}) if it satisfies Axiom~\ref{axiom_ex3}.

\begin{corollary}\label{corr_iso_of_homology_to_ex3}
    In~a~semiexact category, if for any null-sequence $(f,g)$ the~morphism~$m(f,g)$ is an~isomorphism then Axiom~\ref{axiom_ex3} is satisfied.
\end{corollary}

\begin{proof}
Suppose that $i : A \rightarrowtail B$ and $q: B \twoheadrightarrow C$ are a~kernel and a~cokernel such that $(\coker i) \ker q \in \mathcal{N}$.
The~pair of morphisms $(i,q)$ forms a~short exact sequence.
Hence, we have the~diagram
\[
\begin{tikzcd}
    & A \ar[d,tail,"i"] \ar[r,two heads,"\coker \sigma"] & H_{-}(\ker q,\coker i) \\
    \Ker q \ar[r,"\ker q"',tail] \ar[ru,bend left=30,dashed,"\sigma"] & B \ar[r,"\coker i",two heads] \ar[d,"q"',two heads] & \Coker i \\
    H_{+}(\ker q,\coker i) \ar[r,tail,"\ker \tau"'] & C \ar[ru,bend right=30,"\tau"',dashed] &
\end{tikzcd}
\]
whose row and column are null-sequences.
By Lemma~\ref{lemma_m-morphism},
the morphism
$m(\ker q, \coker i)$
makes this diagram commute.
Since $m(\ker q,
\coker i)$ is an
isomorphism,
Axiom~\ref{axiom_ex3} is satisfied:
\[
\begin{tikzcd}
    A \ar[d,two heads,"\coker \sigma"'] \ar[rr,tail,"i"] & & B \ar[dd,two heads,"q"]\\
    H_{-}(\ker q,\coker i) \ar[rd,"m(\ker q {,} \coker i)"',"\cong"] & & \\
    & H_{+}(\ker q,\coker i) \ar[r,tail,"\ker \tau"'] & C
\end{tikzcd}
\]
\end{proof}

The following lemma is \cite[Theorem 2.2.2(ex3b) and (ex3b*)]{Gr92}.
\begin{lemma}
The following hold in a~homological category:\\
(i)
the pullback of a pair
$\cdot\twoheadrightarrow\cdot\leftarrowtail\cdot$
is of the form $\cdot\leftarrowtail\cdot\twoheadrightarrow\cdot$\ ;\\
(ii)
dually,
the pushout of a pair
$\cdot\leftarrowtail\cdot\twoheadrightarrow\cdot$
is of the form $\cdot\twoheadrightarrow\cdot\leftarrowtail\cdot$\ .
\end{lemma}

\begin{theorem}
    An~ex2-category $(\mathcal{C},\mathcal{N}$ is a homological category if~and~only~if
    for any null-sequence $(f,g)$ the~morphism $m(f,g) : H_{-}(f,g) \to H_{+}(f,g)$ is an~isomorphism.
\end{theorem}
\begin{proof}
    By~Corollary~\ref{corr_iso_of_homology_to_ex3}, in~an~ex2-category, if $m(f,g)$ is an~isomorphism, for any null-sequence $(f,g)$, then Axiom~\ref{axiom_ex3} is satisfied.
    Suppose that Axiom~\ref{axiom_ex3} is satisfied.
    Consider a~null-sequence
    \[
    \begin{tikzcd}
        A \ar[r,"f"] \ar[rr,"\mathrm{null}"',bend right=30] & B \ar[r,"g"] & C
    \end{tikzcd}
    \]
    Consider the~pair~$(\ker g, \coker f)$.
    Note that
    $$
    gf = (\im g)\, \overline{g}\, (\coim g)\, (\im f)\, \overline{f} \coim f \in \mathcal{N}.
    $$
    Hence,
    $$
    (\coker(\ker g)) \ker \coker f =
    (\coim g) \im f \in \mathcal{N}
    $$
    because $(\im g)\, \overline{g}$ is an~$\mathcal{N}$-monomorphism and $\overline{f}\, (\coim f)$ is an~$\mathcal{N}$-epimorphism  due to Corollary~\ref{corr_nbimorph_in_the_middle}.
    Therefore, by Axiom~\ref{axiom_ex3}, there exist a~cokernel and a~kernel which make the square
    \[
    \begin{tikzcd}
        \Ker g \ar[r,"\ker g",tail] \ar[dashed,d,two heads,"\pi"'] & B \ar[d,"\coker f",two heads] \\
        H \ar[dashed,r,tail,"\iota"'] & \Coker f
    \end{tikzcd}
    \]
    commute.
    Let us prove that $H$ is isomorphic to both $H_{-}(f,g)$ and $H_{+}(f,g)$.
    By the~universal property of a~kernel (cokernel), there exists a unique $\lambda : H_{-}(f,g) \to H$ such that $\pi = \lambda \coker \sigma$ and there exists a unique $\rho : H \to H_{+}(f,g)$ such that $\iota = (\ker \tau) \rho$.
    We will prove that in fact $\lambda$ and $\rho$ are isomorphisms.
    Hence, $\lambda$ is a cokernel and $\rho$ is a kernel.
    Consider the commutative square:
    \[
    \begin{tikzcd}
        \Ker g \ar[r,"\coker \sigma",two heads] \ar[d,"\ker g"',tail] & \Coker \sigma \ar[d,"\iota \lambda"] \ar[ddr,"x",bend left=25] \\
        B \ar[r,"\coker f"',two heads] \ar[drr,"y"',bend right=25] & \Coker f \ar[dashed,"z",rd] \\
        & & Z
    \end{tikzcd}
    \]
    Consider morphisms $x$ and $y$ such that $x \coker \sigma = (\ker g) y$.
    From~(\ref{diag_homology_explained}) it is clear that
    $$
    yf = y (\ker g) \sigma = x (\coker \sigma) \sigma \in \mathcal{N}.
    $$
    Hence, there exists a unique $z : \Coker f \to Z$ such that $y = z \coker f$ and $x = z \iota \lambda$ due to the identity
    $$
    x \coker \sigma = z \iota \lambda \coker \sigma.
    $$
    Hence, the~left-hand square is a~pushout and, dually, the~right-hand square is a~pullback:
    \[
    \begin{tikzcd}
        \Ker g \ar[r,"\coker \sigma",two heads] \ar[d,"\ker g"',tail] \ar[r, phantom, shift right=4ex, "{\rm PO}" marking] & \Coker \sigma \ar[d,"\iota \lambda",tail] \\
        B \ar[r,"\coker f"',two heads] & \Coker f
    \end{tikzcd}
    \quad\quad
    \quad\quad
    \begin{tikzcd}
        \Ker g \ar[r,"\ker g",tail] \ar[d,"\rho \pi"',two heads] \ar[r, phantom, shift right=4ex, "{\rm PB}" marking] & B \ar[d,"\coker f",two heads] \\
        \Ker \tau \ar[r,"\ker \tau"',tail] & \Coker f
    \end{tikzcd}
    \]
    Note that $\iota \lambda$ is a kernel and $\rho \pi$ is a cokernel and hence
    $$
    \lambda : H_{-} (f,g) \to H,
    \quad\quad
    \quad\quad
    \rho : H \to H_{+} (f,g)
    $$
    are isomorphisms.
\end{proof}

In a homological category, since the left homology of a null-sequence $(f,g)$ is isomoprhic to the right homology, we denote its homology by $H(f,g) := H_{-}(f,g) \cong H_{+}(f,g)$ and for a~long null-sequence~\eqref{diag_long_null_sequence} we denote its \emph{$n$th homology} by $H^n(X) := H^n_{-}(X) \cong H^n_{+}(X)$.

The following lemma is crucial for the
proof of the main theorems
(see \cite{KoLe2024} for a~proof).

\begin{lemma}[The Composition Lemma]\label{composition_lemma}
    Given two morphisms $f$ and $g$
    in a homological category
    such that the composition $gf$
    is defined,
    there exists a null-sequence    \begin{equation}\label{composition_lemma_sequence}
    \begin{tikzcd}
        \bullet \ar[r,"{\rm null}"] & \Ker f \ar[r,"\varphi"] & \Ker(gf) \ar[r,"\psi"] \ar[d, phantom, ""{coordinate, name=Z}] & \Ker g \ar[
            dll,
            "\chi"',
            rounded corners,
            to path={ -- ([xshift=2ex]\tikztostart.east)
            |- (Z) [near end]\tikztonodes
            -| ([xshift=-2ex]\tikztotarget.west) 
            -- (\tikztotarget)}
        ] & \\
        & \Coker f \ar[r,"\varepsilon"] & \Coker(gf) \ar[r,"\omega"] & \Coker g \ar[r,"{\rm null}"] & \bullet
    \end{tikzcd}
    \end{equation}
    which is exact at $\Ker f$, $\Ker(gf)$, $\Coker(gf)$, $\Coker g$.
    Moreover,
    \begin{enumerate}
        \item
        $\varphi$ and $\omega$ are exact morphisms;
        \item
        if $f$ is an exact morphism then {\rm \eqref{composition_lemma_sequence}} is exact at $\Ker g$ and $\psi$ is an exact morphism.
        Dually, if $g$ is an exact morphism then {\rm \eqref{composition_lemma_sequence}} is exact at $\Coker f$ and $\varepsilon$ is an exact morphism.
    \end{enumerate}
\end{lemma}

\section{Lambek's Invariants $\Ker$ and $\Img$}\label{sec-lamb}

In a semiexact category $(\mathcal{C}, \mathcal{N})$, consider the square denoted by $S$:
\begin{equation}\label{sqq}
\begin{tikzcd}
    A \ar[r,"f"] \ar[d,"a"'] \ar[r, phantom, shift right=4ex, "S" marking] & B \ar[d,"b"] \\
    C \ar[r,"g"'] & D
\end{tikzcd}
\end{equation}
In~\cite{Lambek-1964}, Lambek 
defined the invariants $\Img S$ and $\Ker S$ 
of the commutative square $S$ in the category of groups.
In a~semiexact category, these invariants are generalized as follows
(cf. \cite{No1}):
as $\Img S := \Coker \lambda_S$, $\Ker S := \Ker \rho_S$, where $\lambda_S$ 
and $\rho_S$ are the morphisms defined by the universal property of a~pullback and a~pushout:
\begin{equation}\label{def_inv}
\begin{tikzcd}
    & A \ar[dd,"a"'] \ar[rr,"f"] \ar[rd,dashed,"\lambda_S"] & & B \ar[d]\\
    \Img S & & \bullet \ar[r,"l_S",tail] \ar[d] 
\ar[ll,"\coker \lambda_S"{xshift=-30pt},bend left=30,two heads,crossing over] 
\ar[r, phantom, shift right=4ex, "{\rm PB}" marking] & \bullet \ar[d,"\im b",tail]\\
    & C \ar[r] & \bullet \ar[r,"\im g"',tail] & D
\end{tikzcd}
\quad
\begin{tikzcd}
    A \ar[d,"\coim a"',two heads] \ar[r,"\coim f",two heads] 
\ar[r, phantom, shift right=4ex, "{\rm PO}" marking] & \bullet \ar[r] \ar[d] & B \ar[dd,"b"] &\\
    \bullet \ar[d] \ar[r,"r_S"',two heads] & \bullet \ar[rd,dashed,"\rho_S"] & & \Ker S
\ar[ll,"\ker \rho_S"'{xshift=30pt},bend right=30,tail,crossing over] \\
    C \ar[rr,"g"'] & & D &
\end{tikzcd}
\end{equation}

The~following lemma is Lemma~2.10 in \cite{No1}. Its easy proof is identical 
to the~proof in a~Puppe exact category given in~\cite{No1}.

\begin{lemma}\label{ep-mn-triv}
If in diagram~\eqref{sqq} in a~semiexact catory $(\mathcal{C},\mathcal{N})$, $f$ is an~$\mathcal{N}$-epimorphism then $\mathrm{Im\,}S$ is a~null object. If in~\eqref{sqq} $g$ is an~$\mathcal{N}$-monomorphism then $\mathrm{Ker\,}S$ is a~null object. 
\end{lemma}

Consider the diagram of two consecutive squares:
$$
\begin{tikzcd}
    A \ar[r,"f"] \ar[d,"a"'] \ar[r, phantom, shift right=4ex, "S" marking] & B \ar[r,"g"] \ar[d,"b" description] 
\ar[r, phantom, shift right=4ex, "T" marking] & C \ar[d,"c"]\\
    A' \ar[r,"f'"'] & B' \ar[r,"g'"'] & C'
\end{tikzcd}
\eqno{(\ref{maind})}
$$

The following assertion was proved in~\cite{No1} for Puppe exact categories and in~\cite{KoLe2024} for homological categories.

\begin{theorem}\label{lambek_ex_th}
Suppose that in diagram~{\rm \eqref{maind}} in a semiexact category $(\mathcal{C},\mathcal{N})$, 
the morphism $b$ is exact and the compositions $gf$ and $g'f'$ are null.
Then there exists a unique morphism $\Lambda: \Img S \to \Ker T$ such that
\begin{equation}\label{Lambda}
r_T l_S = (\ker \rho_T) \Lambda \coker \lambda_S.
\end{equation}
\end{theorem}

 For brevity, we put $\lambda := \lambda_S$, $l := l_S$, $r := r_T$, and $\rho := \rho_T$.
    Since $b$ is exact, it can be presented as $b = (\im b) \coim b$.
We have the~commutative diagram
    \begin{equation}\label{diagram2}
    \begin{tikzcd}[column sep=large]
        A \ar[rr,"f"] \ar[dd,"a"'] \ar[rd,"\lambda"',dashed] & & B \ar[r,"\coim g",two heads] 
\ar[d,"\coim b"',two heads] & \Coim g \ar[r,"(\im g) \overline{g}"] \ar[d,"t",two heads] & C \ar[dd,"c"]\\
        & M \ar[r, phantom, shift right=4ex, "{\rm PB}" marking] \ar[r,"l",tail] \ar[d,"s"', tail] & \bullet 
\ar[d,"\im b",tail] \ar[r,"r"',two heads] \ar[r, phantom, shift left=4ex, "{\rm PO}" marking] & N 
\ar[rd,"\rho",dashed] & \\
        A' \ar[r,"\overline{f'} (\coim f')"'] & \Img f' \ar[r,"\im f'"',tail] & B' 
\ar[rr,"g'"'] & & C'
    \end{tikzcd}
    \end{equation}
  
We refer to the morphism $\Lambda= \Lambda_{ST} : \Img S \to \Ker T$ constructed 
in Theorem~\ref{lambek_ex_th} as the \emph{Lambek morphism} of the squares $S$ and $T$.

Lemma~\ref{ep-mn-triv} and~\eqref{Lambda} readily imply

\begin{lemma}\label{Lamb-nul}
Suppose that in diagram~{\rm \eqref{maind}} in a semiexact category $(\mathcal{C},\mathcal{N})$, 
the morphism $b$ is exact and, the compositions $gf$ and $g'f'$ are null, and $f$ is an~$\mathcal{N}$-epimorphism or $g'$ is an~$\mathcal{N}$-monomorphism.
Then the~morphism $\Lambda:\Img S\to \Ker T$ is null.
\end{lemma}

The~following theorem was proved for Puppe exact categories in~\cite{No1} and
for homological categories in~\cite{KoLe2024}.

\begin{theorem}[Lambek's Isomorphism Theorem]\label{lam_is}
    Suppose that in diagram~{\rm \eqref{maind}} in a homological category $(\mathcal{C},\mathcal{N})$,
 $b$ is an exact morphism and the rows are exact.
    Then the Lambek morphism $\Lambda: \Img S \to \Ker T$ is an isomorphism.
\end{theorem}

\section{Nomura's Null Sequences}\label{nomura_null}

Throughout the~section, the~ambient category is a~fixed homological category $(\mathcal{C},\mathcal{N})$.
\medskip

Consider the~commutative~diagram of two consecutive squares
    $$
    \begin{tikzcd}
        A \ar[r,"f"] \ar[d,"a"'] \ar[r, phantom, shift right=4ex, "S" marking] & B \ar[r,"g"] \ar[d,"b" description] 
\ar[r, phantom, shift right=4ex, "T" marking] & C \ar[d,"c"]\\
        A' \ar[r,"f'"'] & B' \ar[r,"g'"'] & C'
   \end{tikzcd}
  \eqno{(\ref{maind})}
  $$
with null rows, where $b$ is an exact morphism. 

We will now construct the~first of Nomura's sequences, which is in general just a~null sequence.

We put $H=H\bigl(A \overset{f}{\to} B \overset{g} \to C)$ and $H'=H\bigl(A' \overset{f'}{\to} B' \overset{g'} \to C')$,
where $H(\to\cdot\to)$ stands for the~homology of the~null sequence in parentheses. It is easy to see
that the~vertical arrows in~\eqref{maind} induce a~morphism $h: H\to H'$. 

Indeed, we obtain the~commutative diagram
\begin{equation}\label{square-hom}
\begin{tikzcd}
{\Img f} \arrow[r, "r"] \arrow[d, "\widehat{\widehat{b}}"'] & {\Ker g} \arrow[d, "\widehat{b}"] \\
{\Img f'} \arrow[r,"r' "]                                         & \Ker g'
\end{tikzcd}
\end{equation}
where $r:\Img f \to \Ker g$ and $r':\Img f' \to \Ker g'$ are the~unique morphisms such that $\im f=(\ker g)r$ and $\im f'=(\ker g')r'$; moreover, as above, $\widehat{b}:\Ker g\to \Ker g'$ is the~morphism of the~kernels of the~rows of the~square~$T$ and $\widehat{\widehat{b}}:\Img f\to \Img f'$ is the~morphism
of the~kernels of the~rows of the~square
$$
\begin{tikzcd}
B \arrow[r, "\coker f"] \arrow[d, "b"'] & \Coker f \arrow[d, "\widecheck{b}"]   \\
B' \arrow[r, "\coker f'"']              & \Coker f' \arrow[l, phantom]
\end{tikzcd}
$$
The~natural morphism $h:H\to H'$ of the~homologies of the~rows in~\eqref{maind} is the~morphism of the~cokernels of the~rows in~\eqref{square-hom}:
$$
\begin{tikzcd}[column sep=large]
{\Img f} \arrow[r, "r"] \arrow[d, "\widehat{\widehat{b}}"'] & {\Ker g} \arrow[d, "\widehat{b}"] \arrow[r, "\coker r"] & \Coker r = H \arrow[d, "h"] \\
{\Img f'} \arrow[r, "r'"']                                         & \Ker g' \arrow[r, "\coker r'"']                     & \Coker r'=H'
\end{tikzcd}
$$
or, equivalently, $h$ is the~morphism of the~cokernels in the~diagram
$$
\begin{tikzcd}[column sep= 70pt]
A \arrow[r, "r\overline{f}\coim f"] \arrow[d, "b"'] & \Ker g \arrow[d, "\widehat{b}"] \arrow[r, "\coker(r\overline{f}\coim f)"] & \Coker(r\overline{f}\coim f)= H \arrow[d, "h"] \\
A' \arrow[r, "r'"']                                         & \Ker g' \arrow[r, "\coker(r'\overline{f'}\coim f')"']                     & \Coker(r'\overline{f'}\coim f')=H'
\end{tikzcd}
$$

Another (dual) description of the~morphism $h:H\to H'$ can be obtained if we regard $H$ and $H'$ as the~kernels of the~natural morphisms $\nu_0:\Coker f\to \Coim g$ (equivalently, of $(\im g)\, \overline{g}\nu_0:\Coker f\to C$) and $\nu'_0:\Coker f'\to \Coim g'$ (equivalently, of $(\im g')\, \overline{g'}\nu'_0:\Coker f'\to C'$) respectively.

We put $\nu:=\overline{g}\nu_0$ and $\nu':=\overline{g}\nu'_0$.

Let $\widecheck{b}: \Coker f \to \Coker f'$ be the~morphism of the~cokernels of the~rows 
in~$S$ induced by $b$
and $\eta=\widehat{(\im g)\nu}:\Ker \widecheck{b}\to \Ker c$
be the~morphism of the~kernels in~$S$ induced by $(\im g)\nu$.
Hence the following diagram commutes:
$$
\begin{tikzcd}
& \Ker \widecheck{b} \ar[d,tail,"\ker \widecheck{b}"']
\ar[r,"\eta"] & \Ker c \ar[d,"\ker c",tail]\\
H \ar[r,tail] \ar[d,"h"'] &
\Coker f \arrow[r, "{(\im g)\nu}"] \arrow[d, "\widecheck{b}"'] & C \arrow[d, "c"] \\
H' \ar[r,tail] &
\Coker f' \arrow[r, "(\im g')\nu'"']                                        & C'              
\end{tikzcd}
$$
Then, by Lemma~\ref{lemma_three_by_three}(i), we have $\Ker h = \Ker \eta$.

Denote by~$\xi:=\widehat{\coker f}$
the~morphism of the~kernels of the~columns of the~square $S_1$
induced by $\coker f$:
$$
\begin{tikzcd}
\Ker b \ar[r,"\xi"] \ar[d,tail,"\ker b"'] &
\Ker \widecheck{b} \ar[d,tail,"\ker \widecheck{b}"] \\
B \arrow[r, "\coker f"] \arrow[d, "b"'] \ar[r, phantom, shift right=4ex, "S_1" marking]  
& \Coker f \arrow[d, "\widecheck{b}"]  
\\
B' \arrow[r, "\coker f'"']   &  \Coker f'              
\end{tikzcd}
$$

By the~Composition Lemma applied to the~composition $bf$, we have the~null sequence
\begin{equation}\label{null-bf}
\Ker(bf) \xrightarrow{k} \Ker b \xrightarrow{(\coker f)\ker b}
\Coker f,
\end{equation}
where $k$ is defined by the~equality $(\ker b)k = f\ker (bf)$, which gives a~morphism
$$
k_0: \Ker (bf) \rightarrow \Ker((\coker f) \ker b) = 
\Ker \xi.
$$
Applying the~Composition Lemma to the~composition $\eta\xi$, we get the~null sequence
\begin{equation}\label{seq-etaxi}
\Ker \xi \overset{\mu}{\to} \Ker(\eta\xi) \overset{v}{\to} \Ker \eta
\xrightarrow{(\coker \xi)\ker \eta}  \Coker \xi 
\overset{\rho}{\to} \Coker(\eta\xi).
\end{equation}

We assert that
\begin{equation}\label{short-null}
\Ker(bf) \overset{k}{\to} \Ker b \overset{\eta\xi}{\to} \Ker c
\end{equation}
is a~null sequence, that is, $\eta\xi k$ is null. Indeed, we infer
$$
(\ker\,c)\eta\xi k = (\im g) \nu (\ker \widecheck{b})\xi k 
= (\im g) \nu (\coker f)(\ker b)k 
= (\im g) \nu (\coker f) f(\ker(bf))b \in \mathcal{N}.
$$
Since $\ker c$ is an~$\mathcal{N}$-monomorphism, this gives $\eta\xi k\in\mathcal{N}$.
Consider the~homology
$$
\widehat{H}  = H (\Ker (bf) \overset{k}{\to} \Ker b \overset{\eta\xi}{\to}
\Ker c)
$$
By what was said above, we have the~morphism 
$\mu k_0: \Ker(bf) \to \Ker(\eta\xi)$, and $v\mu k_0$ is null.
This gives a~unique morphism
$$
\widehat{H} = \Coker(\mu k_0) \overset{\alpha}{\longrightarrow} \Ker \eta= \Ker h
$$
such that $v=\alpha \coker(\mu k_0)$. Furthermore, considering the~diagram
\begin{equation}
\label{diagram_xi_eta}
\begin{tikzcd}[
    row sep=large,
    column sep=large
    ]
\Ker A 
\ar[r, phantom, shift right=5ex, "S_0" marking]
\ar[d,tail,"\ker a"'] \ar[r,"\widehat{f}"]
& \Ker b \ar[r, phantom, shift right=5ex, "S_2" marking] \ar[d,tail,"\ker b"description] \ar[r,"\xi"] & \Ker \widecheck{b} \ar[r, phantom, shift right=5ex, "S_3" marking] \ar[d,tail,"\ker \widecheck{b}"description] \ar[r,"\eta"] & \Ker c \ar[d,tail,"\ker c"] \\
A \ar[r, phantom, shift right=5ex, "S" marking] \ar[d,"a"'] \ar[r,"f"description] & B \ar[r, phantom, shift right=5ex, "S_1" marking] \ar[d,"b"description] \ar[r,"\coker f"description,two heads] & \Coker f \ar[d,"\widecheck{b}"] \ar[r,"(\im g)\nu"'] & C \\
\
A' \ar[r,"f'"'] & B' \ar[r,"\coker f'"',two heads] & \Coker f'
\end{tikzcd}
\end{equation}
we have
\begin{equation}\label{identif}
\begin{tikzcd}
\Img S \ar[r,"\cong"',"\Lambda_{S S_1}"] &
\Ker S_1 \ar[r,"\cong"',"\Lambda_{S_2 S_1}^{-1}"] &
\Img S_2 = \Coker \xi.
\end{tikzcd}
\end{equation}
(Here $\Lambda_{AB}$ is the~Lambek (iso)morphism from $\Img A$ into $\Ker B$.)
Since 
$$
(\coker \xi)(\ker\eta)v = (\coker \xi)(\ker\eta)\,\alpha \coker(\mu k_0)  \in \mathcal{N}
$$
and $\coker(\mu k_0)$ is a~cokernel, 
$(\coker \xi)(\ker\eta)\alpha$ is also null, and thus we may assume that 
\[
\beta:=(\coker \xi)\ker\eta:\Ker h \to
\Img S.
\]

Observe that, modulo~\eqref{identif}, the~morphism $\Lambda(\coker \xi)\ker \eta: \Ker \eta \to \Ker T$  is null that is, the~composition
$\Lambda{\Lambda_{S S_1}^{-1}}\Lambda_{S_2 S_1}(\coker \xi) \ker \eta : \Ker \eta \to \Ker T$ is null.  Consider the pushouts
$$
\begin{tikzcd}
\ar[r, phantom, shift right=4ex, "{\rm PO}" marking]
B \arrow[r, "\coker f", two heads] \arrow[d, "\coim b"', two heads] & \Coker f \arrow[d, "\varphi_0", two heads] \\
{\Img b} \arrow[r, "t"', two heads]                                             & V                                                    
\end{tikzcd}
\quad\quad\quad\quad
\text{and}
\quad\quad\quad\quad
\begin{tikzcd}
\ar[r, phantom, shift right=4ex, "{\rm PO}" marking]
B \arrow[r, "\coim g", two heads] \arrow[d, "\coim b"', two heads] & \Coim g \arrow[d, "\varphi", two heads] \\
\Coim b \arrow[r, "s"', two heads]                                            & W                                                  
\end{tikzcd}
$$
of the pairs $(\coker f,\coim b)$
and $(\coim g,\coim b)$.
From the universal property of a pushout there exist unique morphisms $\zeta: V \to W$, 
$\rho_0: V \to \Coker f'$, and $\rho: W \to C'$
that make the following diagram commute:
\[
\begin{tikzcd}[
    row sep=large,
    column sep=large
    ]
B \ar[r, phantom, shift right=5ex, "{\rm PO}" marking] \ar[dd,"b"',bend right=60] \arrow[r, "\coker f", two heads] \arrow[d, "\coim b"'{near end}, two heads] \arrow[rr, "{\textrm{coim}\,g}", two heads, bend left, shift right] & \textrm{Coker}\,f
\arrow[r, "\nu"] \arrow[d, "\varphi_0", two heads] \arrow[dd, "\widecheck{b}"'{near start},bend right=40] & {\textrm{Coim}\, g} \arrow[r, "\im g"] \arrow[d, "\varphi"] & C \arrow[dd, "c"] \\
{\Img b} \arrow[d, "{\im b}"'{near start}, tail] \arrow[r, "t", two heads,crossing over] \arrow[rr, "s"'{near end}, bend right=15, shift right,crossing over]                                                       & V \arrow[d, "\rho_0"{near end},dashed,crossing over] \arrow[r,"\zeta",dashed]                                                                  & W \arrow[rd, "\rho",dashed]                                                           &                   \\
B' \arrow[r, "\coker f' " ', two heads]                                                     & \Coker f' \arrow[rr, "(\im g')\nu'"',dashed]                                                                        &                                                                                & C'               
\end{tikzcd}
\]
Consider the morphism of kernels
$$
\Ker S_1  = \Ker \rho_0 \overset{\widehat{\zeta}}{\longrightarrow} \Ker \rho
= \Ker T
$$
induced by the morphism $\zeta : V \to W$. By Lambek's Isomorphism Theorem, $\Lambda_{S S_1}:\Img S \cong \Ker S_1$. Due to the~naturality of the~Lambek morphism, we have $\Lambda= \widehat{\zeta}\Lambda_{S S_1}$. We infer
\begin{align*}
(\ker \rho)\,\Lambda{\Lambda_{S S_1}^{-1}}\Lambda_{S_2 S_1}(\coker \xi) \ker \eta 
&= (\ker \rho)\,\widehat{\zeta}\,\Lambda_{S_2,S_1}\,
(\coker \xi) \ker \eta \\
&= \zeta\, \varphi_0\, (\ker \widecheck{b})\,
\ker \eta \\
&= \varphi\, \nu\, (\ker \widecheck{b})\,
\ker \eta \\
&= \varphi\,\nu\, (\ker \widecheck{b})\,
\left(\ker \left( (\im g)\nu \ker \widecheck{b}
\right)\right) \\ 
&= \varphi\,\nu\, (\ker \widecheck{b})\,
\left(\ker \left( \nu \ker \widecheck{b}
\right)\right) \in \mathcal{N}.
\end{align*}
Therefore, since $\ker \rho$ is an~$\mathcal{N}$-monomorphism,
the~composition 
$\Lambda{\Lambda_{S S_1}^{-1}}\Lambda_{S_2 S_1}(\coker \xi) \ker \eta$ 
is null, q.e.d. 

Thus, we have constructed the null sequence
\[
\begin{tikzcd}
H\left( \Ker(bf) \to \Ker b \to \Ker c \right) \ar[r,"\alpha"] & \Ker\left( H \to H' \right) \ar[r, "\beta"] & \Img S \ar[r,"\Lambda"] &
\Ker T\,.
\end{tikzcd}
\]

Dually, there exist morphisms
$\eta':\Coker a\to\Coker \widehat{b}$
and
$\xi':\Coker \widehat{b}\to\Coker b$
which make the diagram 
\begin{equation}
\label{diag_with_T_0}
\begin{tikzcd}[
    row sep=large,
    column sep=large
    ]
&
\Ker g \ar[r,"\ker g"]
\ar[d,"\widehat{b}"'] &
B \ar[r,"g"] \ar[d,"b"] &
C \ar[d,"c"]
\\
A' \ar[d,"\coker a"'] \ar[r] &
\Ker g'
\ar[d,"\coker \widehat{b}"']
\ar[r,"\ker g'",tail] &
B'
\ar[r, phantom, shift right=5ex, "T_0" marking]
\ar[d,"\coker b"',two heads]
\ar[r,"g'"] & C'
\ar[d,"\coker c",two heads] 
\\
\Coker a \ar[r,"\eta'"'] &
\Coker \widehat{b}
\ar[r,"\xi'"',two heads] &
\Coker b  \ar[r,"\widecheck{g'}"'] 
& \Coker c
\end{tikzcd}
\end{equation}
commute. Here $\widehat{b}:\Ker g\to \Ker g'$ is the~natural morphism such that $(\ker\,g')\widehat{b}=b\ker\,g$.
Hence, there exist morphisms $\alpha'$ and
$\beta'$ that complete Nomura's
null-sequence 
$$
\begin{tikzcd}[column sep=small]
     & H\left( \Ker(bf) \to \Ker b \to \Ker c \right) \ar[r,"\alpha"] \ar[d, phantom, ""{coordinate, name=Z}] & \Ker\left( H \to H' \right) \ar[r, "\beta"] & \Img S 
    \ar[
        dll,
        "\Lambda"',
        rounded corners,
        to path={ -- ([xshift=2ex]\tikztostart.east)
        |- (Z) [near end]\tikztonodes
        -| ([xshift=-2ex]\tikztotarget.west) 
        -- (\tikztotarget)}
    ]\\
    & \Ker T \ar[r,"\beta'"'] & \Coker\left( H \to H' \right) \ar[r,"\alpha'"'] & H\left( \Coker a \to \Coker b \to \Coker(g' b) \right)
\end{tikzcd}
\eqno{(\ref{diagram_nomura_exact_seq})}
$$
in a~homological category.

\medskip

We have

\begin{theorem}\label{th_exact_1}
The~following exactness properties hold for diagram~\eqref{diagram_nomura_exact_seq} with the~morphism~$b$ exact:

(i) If $f$ is exact then $\alpha$ is an $\mathcal{N}$-monomorphism.
If in~addition $(\coker f)\ker b$ is exact then
 $\alpha=\ker\beta$ and sequence~\eqref{diagram_nomura_exact_seq}
is exact at~$\Ker(H\to H')$.
                   
(ii) If $g$ and $\eta$ are exact then \eqref{diagram_nomura_exact_seq} is exact at~$\Img S$
and $\Lambda$ is an~exact morphism.

Dually,

(i\,$'$) If $g'$ is exact then $\alpha'$ is an $\mathcal{N}$-epimorphism.
If in~addition $(\coker b)\ker g'$ is exact then
 $\alpha'=\coker \beta'$ and sequence~\eqref{diagram_nomura_exact_seq}
is exact at~$\Coker(H\to H')$.
                   
(ii\,$'$) If $f'$ and $\eta'$ are exact then \eqref{diagram_nomura_exact_seq} is exact at~$\Ker T$
and $\Lambda$ is an~exact morphism.
\end{theorem}

\begin{proof}
We begin with some general considerations. 

Let
$$
\begin{tikzcd}
P \arrow[d, "y"', tail] \arrow[r, "x", tail]            & \Ker C \arrow[d, "{\mathrm{ker\,}c}", tail] \\
{\mathrm{Im\,}g} \arrow[r, tail] \arrow[r, "{\mathrm{im\,}g}"'] & C                                                     
\end{tikzcd}
$$
be a~pullback. Then there is a~unique morphism $\varkappa:\mathrm{Ker\,}\widecheck{b} \to P$ such that $\eta=\varkappa x$ and
$t\ker \widecheck{b}=\varkappa y$; there is also a~unique morphism $\gamma:\mathrm{Ker\,} b \to P$ such that
$x\gamma = x\varkappa \xi$ (and thus $\gamma = \varkappa \xi$) and 
$y\gamma= \nu(\coker f)\ker b= \overline{g}(\coim g)\ker b$. We obtain the~commutative diagram
$$
\begin{tikzcd}[
    row sep=large,
    column sep=large
    ]
& \Ker b \ar["\ker b"',d,tail] \ar["\xi"',r]  \ar[rr,bend left=20,"\gamma"] 
& \Ker \widecheck{b} \ar[rr,bend left=20,"\eta",crossing over] \ar[r,"\kappa"',dashed] \ar[d,tail,"\ker \widecheck{b}"',crossing over] & P \ar[r, phantom, shift right=5ex, "{\rm PB}" marking] \ar["x",r,tail] \ar["y"',d,tail] & \Ker c \ar[d,tail,"\ker c"] \\
A \ar[r,"f"] \ar[d,"a"'] & B \ar[d,"b"'] \ar[r,"\coker f"',two heads] & \Coker f \ar[d,"\widecheck{b}"] \ar[r,"\nu"',dashed] & \Img g \ar[r,"\im g"',tail] & C  \\
A' \ar[r,"f'"'] & B' \ar[r,"\coker f'"',two heads] & \Coker f'
\end{tikzcd}
$$
We have $\gamma=\varkappa \xi$ and $\eta\xi = x\gamma$. It is not hard to see that the~canonical morphism
$\rho:\Coker \xi\to \Coker(\eta\xi)$ in~\eqref{seq-etaxi} is the~composition of 
the~two canonical morphisms
$$
\Coker \xi \overset{\varepsilon}{\to} \Coker \gamma= \Img S_2 S_3 
\overset{\tau}{\rightarrowtail} \Coker (\eta\xi). 
$$
Here the~morphism~$\tau$ is a~kernel since $x$ is exact, $\ker x$ is null, and so the~sequence
$$
\Ker x \xrightarrow{(\coker \gamma)\ker\varkappa} 
\Coker \gamma \overset{\tau}{\rightarrowtail} \Coker(\eta\xi)
$$
is exact at~$\Coker \gamma$ by the~Composition Lemma applied to the~composition $x\alpha$.

(i) Suppose that $f$ is exact. Then sequence~\eqref{null-bf} is exact at~$\Ker b$ and
the~morphism $k:\Ker(bf)\to \Ker b$ is exact and so we see that 
$k_0=\coim k$ is a~cokernel. Thus, since sequence~\eqref{seq-etaxi} is exact 
at~$\Ker(\eta\xi)$, we have $\coker (\mu k_0)=\coker \mu = \coim v$,
the~morphism $\alpha=(\im v) \overline{v}:\widehat{H} = \Coker(\mu k_0) \to \Ker \eta= \Ker h $ 
is $\mathcal{N}$-monic. 
 
If in~addition $(\coker f)\ker b$ is exact then $\xi$ is exact, by~the~Composition Lemma, \eqref{seq-etaxi} is exact
at~$\Ker \eta$, $v$ is exact and then $\alpha=\ker\beta$ and sequence~\eqref{diagram_nomura_exact_seq}
is exact at~$\Ker(H\to H')$.

(ii) Suppose that $\eta$ and $g$ are exact. Then sequence~\eqref{seq-etaxi} is exact at~$\Coker \xi$
and $\rho$ is exact, that is, $\im [(\coker \xi)\ker\eta] = \ker\rho = \ker\varepsilon$, 
and $\varepsilon$ is exact. Since $g$ is exact, we have the~Lambek isomorphism 
$\Lambda_{S_2 S_3,T}:\Img S_2 S_3 \xrightarrow{\cong} \Ker T$, by which
$$
\Lambda=\Lambda_{ST}= \Lambda_{S_2 S_3,T} \,\varepsilon\, \Lambda_{S S_1}^{-1} \Lambda_{S_2 S_1}.  
$$
Thus, $\Img \beta = \Ker \Lambda$ and $\Lambda$ is exact.
\end{proof}

\smallskip

The~following theorem is a~version of Corollary~A2 to Theorem~A in~\cite{No1}. However, we prove it directly because its hypotheses in a~homological category are weaker than those given by Theorem~\ref{th_exact_1}.

\begin{theorem}\label{corol1}
Suppose that the~row $A'\to B' \to C'$
is exact and the~morphism $b$ is a~kernel in \eqref{maind}.
Then the~sequence
$$
\begin{tikzcd}
\bullet \ar[r,"\mathrm{null}"] & H(A \to B \to C) \ar[r,"\beta",tail] & \Img S \ar[r,"\Lambda", two heads] & \Ker T  \ar[r,"\mathrm{null}"] & \bullet
\end{tikzcd}
$$
is exact, that is, $\beta = \ker \Lambda$ and $\Lambda=\coker \beta$.
Dually, if the~row $A\to B \to C$
is exact and $b$ is a~cokernel in \eqref{maind}
then we have the~sequence 
$$
\begin{tikzcd}
\bullet \ar[r,"\mathrm{null}"] &  \Img S \ar[r,"\Lambda",tail] & \Ker T 
\ar[r,"\beta' ", two heads] & H(A' \to B' \to C')  \ar[r,"\mathrm{null}"] & \bullet
\end{tikzcd}
$$
is exact, that is, 
$\Lambda = \ker \beta'$ and $\beta' = \coker \Lambda$.
\end{theorem}

\begin{proof}
Consider the~commutative diagram
$$
\begin{tikzcd}[
    row sep=20pt,
    column sep=25pt]
\ar[r, phantom, shift right=4.5ex, "S" marking]
A \arrow[r, "f"] \arrow[d, "a"'] & B  
\arrow[r, "\coker f",two heads] 
\ar[r, phantom, shift right=4.5ex, "S_1" marking]
\arrow[d, "{\im b=b}"description, tail] \arrow[rr, "\coim g", bend left, two heads] & \Coker f \arrow[d, "\widecheck{b}"'] \arrow[r, "t_0",two heads]                          & \Coim g \arrow[r, "{(\im g)\overline{g}}"] \arrow[ld, "v_0"] & C \arrow[d, "c"] \\
A' \arrow[r]                    & B' \arrow[r,two heads]                                                                                                      & \Coker f'=\Coim g' \arrow[rr, "{(\im g')\overline{g'}}"'] &                                                                                  & C'
\end{tikzcd}
$$
where $t_0:\Coker f \to \Coim g$ is the~unique morphism
such that $t_0 \coker f = \coim g$ and
$v_0:\Coim g\to \Coim g'=\Coker f'$
is the~natural morphism of coimages such that $c(\im g)\overline{g}=
(\im g')\overline{g'}\,v_0$.

We have
$$
\Ker t_0 = H(A\to B \to C),
\quad\quad
\Img S \cong \Ker S_1 = \Ker \widecheck{b},
\quad\quad  
\Ker v_0 = \Ker T.
$$
Since $t_0$ is a~cokernel, the~Composition Lemma applied to the~composition
$\widecheck{b}= v_0 t_0$ gives the~exact sequence
$$
\begin{tikzcd}
\bullet \ar[r,"\mathrm{null}"] &  \Ker t_0 \ar[r,tail] & \Ker \widecheck{b} 
\ar[r, two heads] & \Ker v_0  \ar[r,"\mathrm{null}"] & \bullet \,\,,
\end{tikzcd}
$$
that is,
$$
\begin{tikzcd}
\bullet \ar[r,"\mathrm{null}"] & H(A \to B \to C) \ar[r,"\beta",tail] & \Img S \ar[r,"\Lambda", two heads] & \Ker T  \ar[r,"\mathrm{null}"] & \bullet \,\,.
\end{tikzcd}
$$

The~second assertion of the~theorem follows by duality.
\end{proof}

\medskip

We have the~commutative diagram
\begin{equation}\label{diag-full}
\begin{tikzcd}
\Ker a \ar[d,"\ker\,a"'] \ar[r,"\widehat{f}"] & \Ker b \ar[d,"\ker\,b"'] \ar[r,"\widehat{g}"] & \Ker c \ar[d,"\ker\,c"] 
\\
    A \ar[r,"f"] \ar[d,"a"'] \ar[r, phantom, shift right=4ex, "S" marking] & B \ar[r,"g"] \ar[d,"b"description] \ar[r, phantom, shift right=4ex, "T" marking] & C \ar[d,"c"]\\
    A' \ar[r,"f'"']    \ar[d,"\coker a"'] 
    & B' \ar[r,"g'"'] \ar[d,"\coker b"'] 
    & C' \ar[d,"\coker c"]    \\
   \Coker a \ar[r,"\widecheck{f'}"'] 
   & \Coker b \ar[r,"\widecheck{g'}"'] & \Coker c 
\end{tikzcd}
\end{equation}
in which $\widehat{f}:\Ker a\to \Ker b$ and 
$\widehat{g}:\Ker b\to \Ker c$
(respectively, $\widecheck{f'}:\Coker a\to \Coker b$ and 
$\widecheck{g'}:\Coker b\to \Coker c$) are the~natural morphisms of the~kernels (respectively, of the~cokernels) in the~squares~$S$ and $T$. Let $H(\Ker a\to \Ker b\to \Ker c)$ be the~homology of the~first row in~\eqref{diag-full} at~$\Ker b$ and let $H(\Coker a\to \Coker b\to \Coker c)$ be the~homology of the~last row in~\eqref{diag-full} at~$\Coker b$.

Now, consider the~commutative diagram
\begin{equation}\label{diag-t3}
\begin{tikzcd}[
    column sep=35pt,
    row sep=25pt]
\Ker a \arrow[r, "\lambda"] \arrow[d, "\ker a"', tail] & \Ker \xi 
\ar[r, phantom, shift right=5ex, "S_4" marking] \arrow[r, "\mu",tail]
 \arrow[rr, "{\ker \xi}", bend left]\arrow[d, "w"', tail] & \Ker(\eta\xi) 
\ar[r, phantom, shift right=5ex, "S_5" marking] \arrow[r, "\ker (\eta\xi)", tail] \arrow[d, "{\widehat{\ker\,b}}"description, tail] & \Ker b \arrow[d, "{\ker\,b}", tail] \\
A \arrow[r, "{\overline{f}\coim f}"']                                & {\Img f} \arrow[r, "z"', tail] \arrow[rr, "{\im f}"', bend right]                       & \Ker g \arrow[r, "{\ker\,g}"', tail]                                                                                                                    & {B}                        
\end{tikzcd}
\end{equation}
In~\eqref{diag-t3}, the~morphism $z:\Img f\to \Ker g$ is defined by the~equality $\im f=(\ker\,g)z$; $\widehat{\ker\,b}$
is the~morphism of the~kernels of the~rows of the~square
$$
\begin{tikzcd}
\Ker b \arrow[r, "\eta\xi"] \arrow[d, "{\ker\,b}"', tail] & \Ker c \arrow[d, "{\ker\,c}", tail] \\
B \arrow[r, "g"']                                                    & C     
\end{tikzcd}
$$
and the~morphism $w:\Ker \xi\to \Img f$ is defined as follows.
Since $(\coker f)(\ker b) \ker \xi = (\ker \widecheck{b}) \xi (\ker \xi)$
is null and $\im f=\ker(\coker f)$, there exists a~unique morphism
$w:\Ker \xi\to \Img f$ such that $(\ker b)(\ker \xi)=(\im f)w$. The~commutativity of the~left-hand square in~\eqref{diag-t3} follows from the~fact that
$$
(\im f) w\lambda = (\ker\,b)(\ker\xi)\lambda= (\ker\,b)\widehat{f}
=f(\ker\,a) = (\im f) \overline{f}\,(\coim f)(\ker\,a)
$$
and $\im f$ is an~$\mathcal{N}$-monomorphism. 
Since $\mu$ is exact, the~Composition Lemma gives the~exact sequence
\begin{equation}\label{mu-lambda}
\begin{tikzcd}
\bullet \ar[r,"\mathrm{null}"] &
\Coker \lambda \ar[r,"p"] &
\Coker(\mu\lambda) \ar[r,"q"] &
\Coker \mu \ar[r,"\mathrm{null}"] & \bullet.
\end{tikzcd}
\end{equation}

By~Lemma~\ref{PB}, the~square $S_4 S_5$ in~\eqref{diag-t3} is a~pullback.
Indeed, we have the~commutative diagram
$$
\begin{tikzcd}
\Ker \xi \arrow[r, "\ker \xi", tail] \arrow[d, "w"', tail] & \Ker b \arrow[r, "\xi"] \arrow[d, "\ker b"', tail] & \Ker \widecheck{b} \arrow[d, "\ker\widecheck{b}"] \\
{\Img f} \arrow[r, "{\im f}"', tail]                    & B   \arrow[r, "\coker f"', two heads]                 & \Coker f                           
\end{tikzcd}
$$
We have
$$
H(\Ker a\to \Ker b\to \Ker c) = \Coker\left( \Ker a\overset{\mu\lambda}{\to}\Ker(\eta\xi)\right).
$$
Since the~natural morphism $v:\Ker(\eta\xi)\to \Ker \eta$
satisfies $v\mu\in \mathcal{N}$, we have $v=v'\coker \mu$ for some arrow $v':\Coker \mu\to\Ker \eta$. This gives the~morphism 
$\varkappa=v'q:\Coker(\mu\lambda)\to \Ker \eta$. Since
$\beta\alpha$ is null and $v=\alpha\coker(\mu k_0)$, we get 
$\beta v = \beta v'\coker \mu\in \mathcal{N}$, and since $\coker \mu$ is an~$\mathcal{N}$-epimorphism, this gives $\beta v'\in \mathcal{N}$ and hence $\beta\varkappa\in \mathcal{N}$.

Thus, using sequences~\eqref{mu-lambda} and~\eqref{diagram_nomura_exact_seq}
and duality, we obtain the~null-sequence
\begin{equation}\label{seq-2.0}
\begin{tikzcd}[column sep=small]
    \ar[d, phantom, ""{coordinate, name=Z}] \Img S_0 \ar[r,"p"] & H(\Ker a\to \Ker b\to \Ker c) 
    \ar[r,"\varkappa"] & \Ker\left( H \to H' \right) \ar[r, "\beta"] & \Img S 
    \ar[
        dlll,
        "\Lambda"',
        rounded corners,
        to path={ -- ([xshift=2ex]\tikztostart.east)
        |- (Z) [near end]\tikztonodes
        -| ([xshift=-2ex]\tikztotarget.west) 
        -- (\tikztotarget)}
    ]
    \\
    \Ker T \ar[r,"\beta'"'] & \Coker\left( H \to H' \right) \ar[r,"\varkappa'"'] & H(\Coker a\to \Coker b\to \Coker c ) \ar[r,"p'"']
    & \Ker T_0
\end{tikzcd}
\end{equation}

If $f$ is exact then, by Lambek's theorem, $\Img S_0\cong \mathrm{Ker S}$; dually, if $g'$ is exact then $\Ker T_0\cong \Img T$. Thus, if both $f$ and $g'$ are exact then we can replace~\eqref{seq-2.0} by
the~sequence
$$
\begin{tikzcd}[column sep=small]
    \ar[d, phantom, ""{coordinate, name=Z}] \Ker S \ar[r,"p_1"] & H(\Ker a\to \Ker b\to \Ker c) 
    \ar[r,"\varkappa"] & \Ker\left( H \to H' \right) \ar[r, "\beta"] & \Img S 
    \ar[
        dlll,
        "\Lambda"',
        rounded corners,
        to path={ -- ([xshift=2ex]\tikztostart.east)
        |- (Z) [near end]\tikztonodes
        -| ([xshift=-2ex]\tikztotarget.west) 
        -- (\tikztotarget)}
    ]
    \\
    \Ker T \ar[r,"\beta'"'] & \Coker\left( H \to H' \right) \ar[r,"\varkappa'"'] & H( \Coker a\to \Coker b\to \Coker c) \ar[r,"p_1'"']
    & \Img T
\end{tikzcd}
\eqno{\eqref{seq-2}}
$$

Using~Theorem~\ref{th_exact_1}, we obtain the~following analog of~Nomura's
Theorem~B in~\cite{No1}:

\begin{theorem}\label{exact-seq-2}
Suppose that the~morphisms~$b$, $f$, and $g'$ in~\eqref{maind} are exact. 
The~following exactness properties hold for~\eqref{seq-2}:

(i) $p_1=\ker \varkappa$, $p'_1=\coker \varkappa'$. Moreover, if $(\coker f) \ker b$ is exact then sequence~\eqref{seq-2} is exact at~$\Ker(H\to H')$ and  $\varkappa$ is an~exact morphism; if $(\coker b) \ker g'$ is exact then sequence~\eqref{seq-2} is exact at $\Coker(H\to H')$ and $\varkappa'$ is an~exact morphism.
                   
(ii) If $g$ and $\eta$ are exact then \eqref{seq-2} is exact at~$\Img S$ and $\Lambda$ is an~exact morphism. If $f'$ and $\eta'$ are exact then \eqref{seq-2} is exact at~$\Ker T$
and $\Lambda$ is an~exact morphism.
\end{theorem}

\begin{proof}
(i) The~relation $p_1=\ker \varkappa$ follows from the~exact sequence~\eqref{mu-lambda}. The~relations 
\begin{equation}\label{rels}
v=\alpha \coker(\mu k_0), \quad v=v'\coker \mu, \quad \varkappa= v'q,
\end{equation}
and the~fact that 
$\coker(\mu k_0)$, $\coker \mu$, and $q$ are cokernels implies that 
$$
\im \varkappa =  \im v' = \im v = \im \alpha. 
$$
Suppose that if the~morphism $(\coker f)\ker b$ is exact. Then the~exactness of~\eqref{seq-2} at~$\Ker(H\to H')$ follows from the~exactness of~\eqref{diagram_nomura_exact_seq} at~$\Ker(H\to H')$  by Theorem~\ref{th_exact_1}(i). Also by Theorem~\ref{th_exact_1}(i), $\alpha$ is a~kernel, and relations~\eqref{rels} also imply that $\varkappa$
is exact. The~remaining assertions of item~(i) follow by duality.

Item~(ii) follows directly from Theorem~\ref{th_exact_1}(ii,\,ii$'$)
 \end{proof}
 
\smallskip

The~following theorem is a~version of Corollary~B2 in~\cite{No1}. We give a~direct proof because it needs less assumptions than the~corresponding
corollary to Theorem~\ref{exact-seq-2}.

\begin{theorem}
If, in~diagram~\eqref{maind}, $f$ is exact, $f'$ and $h:H\to H'$ are $\mathcal{N}$-monomorphisms then the~sequence $\Ker a \overset{\widehat{f}}{\to} \Ker b \overset{\widehat{g}}{\to} \Ker c$
is exact, that is,
$\im \widehat{f} = \ker \widehat{g}$. 

Dually, if, in~\eqref{maind}, $g'$ is exact and $g$ and $h:H\to H'$ are $\mathcal{N}$-epimorphisms then the~sequence
$\Coker a \overset{\widecheck{f}}{\to} \Coker b \overset{\widehat{g}}{\to} \Coker c$ is exact.
\end{theorem}

\begin{proof}
Suppose that $x:X\to \Ker b$ is a~morphism such that $\widehat{g}x\in \mathcal{N}$. We will prove that $x=(\im \widehat{f}) \widetilde{x}$ for some unique~$\widetilde{x}$.
Obviously, we may assume without loss of generality that $x$ is a~kernel, and thus $x=\im x$.

Since 
$$
g(\ker b)x=(\ker c)\widehat{g}x\in \mathcal{N},
$$
there exists a~morphism $z:X\to \Ker g$ such that $(\ker b)x=(\ker g)z$. Moreover,
$$
(\ker g')\widehat{b}z=b(\ker g)z= b(\ker b)x \in \mathcal{N},
$$
and since $\ker g$ is an~$\mathcal{N}$-monomorphism, we have $\widehat{b}z\in \mathcal{N}$. Since
$$
h(\coker r)z = (\coker r_1)\widehat{b}z\in \mathcal{N}
$$
and $h$ is an~$\mathcal{N}$-monomorphism, we have $(\coker r)z\in \mathcal{N}$. We have $r=\ker(\coker r)$, and hence there is a~morphism
$\pi:X\to \Img f$ with $z=r\pi$. Since $z$ is a~kernel, so is $\pi$.
We have (see~\eqref{square-hom})
$$
r' \widehat{\widehat{b}}\pi = \widehat{b} r\pi= \widehat{b} z \in \mathcal{N}.
$$
Since $r'$ is an~$\mathcal{N}$-monomorphism, this implies that $\widehat{\widehat{b}}\pi \in\mathcal{N}$.

Since $f$ is exact, $f=(\im f) \coim f$. Consider the~pullback (along the~kernel $\pi$)
$$
\begin{tikzcd}
Y \arrow[r, "y_2"] \arrow[d, "y_1"'] & X \arrow[d, "\pi"]                \\
A \arrow[r, "{\coim f}"']    & {\Img f=\Coim f}
\end{tikzcd}
$$
Observe that $\widehat{\widehat{b}}= b \im f$.
We have
$$
f'a y_1= bf y_1 = b(\im f)(\coim f)y_1 
= b(\im f)\pi y_2 = \widehat{\widehat{b}}\pi y_2 \in \mathcal{N}.
$$
Since $f'$ is a~monomorphism, it follows that $a y_1\in \mathcal{N}$. Hence there exists
a~morphism $y_0:Y\to \Ker a$ with $y_1=(\ker a)y_0$. We infer
$$
(\ker b)x y_2= (\ker g)z  y_1 =
(\ker g)r\pi y_1 = (\im f)\pi y_1 
= (\im f)(\coim f)y_1 = f y_1 =
f(\ker a) y_0 =
(\ker b)\widehat{f} y_0.
$$
Since $\ker b$ is an~$\mathcal{N}$-monomorphism, this gives
$$
x y_2 = \widehat{f} y_0.
$$
Since $x$ is a~kernel and $y_2$ is a~cokernel,
$$
x = \im(x y_2) = (\im \widehat{f})
\im(\widehat{f} y_0).
$$ 
We put 
$$
\widetilde{x} = \im(\widehat{f} y_0).
$$
The~condition $x=(\im \widehat{f}) \widetilde{x}$ defines
$\widetilde{x}$ uniquely because $\im \widehat{f}$ is a~monomorphism. 

The~first assertion of the~theorem is proved.

The~second assertion is proved by duality.
\end{proof}

Let
\begin{equation}\label{5ll}
\begin{tikzcd}
A \arrow[r, "f"] \arrow[d, "a"'] & B \arrow[r, "g"] \arrow[d, "b"'] & C \arrow[r, "h"] \arrow[d, "c"'] & D \arrow[r, "k"] \arrow[d, "d"] & E \arrow[d, "e"] \\
A'  \arrow[r, "f'"']             & B' \arrow[r, "g'"']              & C' \arrow[r, "h'"']              & D' \arrow[r, "k'"']             & E'         \end{tikzcd}
\end{equation}
be a~commutative diagram with null rows. Diagram~\eqref{5ll} induces two zero sequences 
$$
\begin{tikzcd}
\Ker c \arrow[r, "\widehat{h}"] & \Ker d \arrow[r, "\widehat{k}"] & \Ker e
\end{tikzcd}
$$
and
$$
\begin{tikzcd}
\Coker a \arrow[r, "\widecheck{f'}"] & \Coker b \arrow[r, "\widecheck{g'}"] & \Coker c \,.
\end{tikzcd}
$$
This gives two homology objects $H(\Ker c\to \Ker d \to \Ker e)$ and $H(\Coker a\to \Coker b \to \Coker c)$.

\medskip
We have the~following theorem, which is a~generalized version of the~five lemma in a~homological category whose Puppe exact counterpart is Theorem~7 in \cite{No1}.

\begin{theorem}\label{5l}
If the~upper row in~\eqref{5ll} is exact at~$C$ and~$D$, the~lower row is exact at~$B'$ and $C'$, and the~morphisms $h$, $c$, and $g'$ are exact then 
$$
H(\Ker c\to \Ker d \to \Ker e) \cong H(\Coker a\to \Coker b \to \Coker c) \,.
$$
\end{theorem}

\begin{proof}
We have the~commutative diagram
$$
\begin{tikzcd}[row sep=25pt]
&                                                       & \Ker c
\ar[r, phantom, shift right=4.5ex, "S_{\rm I}" marking]
\arrow[d, "{\ker\,c}"', tail] \arrow[r, "\widehat{h}"] & \Ker d \arrow[d, "{\ker\,d}"description,tail] 
\ar[r, phantom, shift right=4.5ex, "S_{\rm II}" marking]
\arrow[r, "\widehat{k}"] & \Ker e \arrow[d, "{\ker\,e}", tail] \\
A \arrow[r, "f"] \arrow[d, "a"']                      & B \arrow[r, "g"] 
\ar[r, phantom, shift right=4.5ex, "S_{\rm III}" marking]
\arrow[d, "b"']                      & C 
\ar[r, phantom, shift right=4.5ex, "S_{\rm IV}" marking]
\arrow[r, "h"description] \arrow[d, "c"description]                                   & D \arrow[r, "k"'] \arrow[d, "d"]                                   & E \arrow[d, "e"]                         \\
A' 
\ar[r, phantom, shift right=4.5ex, "S_{\rm V}" marking]
\arrow[r, "f'"] \arrow[d, "\coker a"', two heads] & B' \arrow[r, "g'"description]
\ar[r, phantom, shift right=4.5ex, "S_{\rm VI}" marking]
\arrow[d, "\coker b"description,two heads] & C' \arrow[r, "h'"'] \arrow[d, "\coker c", two heads]               & D' \arrow[r, "k'"']                                               & E'                                       \\
\Coker a \arrow[r, "\widecheck{f'}"']           & \Coker b \arrow[r, "\widecheck{g'}"']           & \Coker c                                                &                                                                   &                                         
\end{tikzcd}
$$
Theorem~\ref{lam_is} about the~Lambek isomorphism gives
\begin{equation}\label{isom16}
\Img S_{\rm I} \cong \Ker S_{\rm IV} \cong \Ker S_{\rm III}
\cong \Ker S_{\rm VI}.
\end{equation}
By Theorem~\ref{corol1}, we have the~exact sequences with exact morphisms
$$
\begin{tikzcd}
\bullet \ar[r,"\mathrm{null}"] & H(\Ker c\to \Ker d \to \Ker e) \ar[r,"\beta",tail] & \Img S_{\rm I} \ar[r,"\Lambda_{S_{\rm I} S_{\rm II}}", two heads] & \Ker S_{\rm II}  \ar[r,"\mathrm{null}"] & \bullet
\end{tikzcd}
$$
and
$$
\begin{tikzcd}
\bullet \ar[r,"\mathrm{null}"] &  \Img S_{\rm V} \ar[r,"\Lambda_{S_{\rm V} S_{\rm VI}}",tail] & \Ker S_{\rm VI}
\ar[r,"\beta' ", two heads] & H(\Coker a\to \Coker b \to \Coker c)  \ar[r,"\mathrm{null}"] & \bullet
\end{tikzcd}
$$
Since $\ker\,e$ is an~$\mathcal{N}$-monomorphism, by~Lemma~\ref{ep-mn-triv}, $\Ker S_{\rm II}$ is null, and so $\Lambda_{S_{\rm I} S_{\rm II}}$ is null. Therefore,  
$H(\Ker c \to \Ker d \to \Ker e) \cong \Img S_{\rm I}$. Moreover, since $\coker a$ is an~$\mathcal{N}$-epimorphism, by Lemma~\ref{ep-mn-triv}, $\Img S_{\rm V}$ is null, and so $\Lambda_{S_{\rm V} S_{\rm VI}}$ is null. Hence, 
$H(\Coker a\to \Coker b \to \Coker c) \cong \Ker S_{\rm VI}$. Reckoning with~\eqref{isom16}, we obtain the~desired
isomorphism
$$
H(\Ker c\to \Ker d \to \Ker e) \cong H(\Coker a\to \Coker b \to \Coker c) \,.
$$
\end{proof}

\section*{Ackowledgment}

The work of Ya.~Kopylov was carried out in the framework of the State Task to the Sobolev
Institute of Mathematics (Project FWNF--2026--0026).

\section*{Declarations}

$\bullet$ \textbf{Conflict of interest/Competing interests:} The authors declare that
there is no conflict of interest.


$\bullet$ \textbf{Availability of data and materials:} Data sharing not applicable to this
article as no datasets were generated or analyzed during the current study.

$\bullet$ \textbf{Authors’ contributions:} The authors contributed equally to this work.

\end{document}